\documentclass{amsart}

\usepackage{amssymb}

\newcommand{\spec}{\operatorname{spec}}

\newtheorem{theorem}{Theorem}[section]

\theoremstyle{definition}

\theoremstyle{remark}
\newtheorem{remark}[theorem]{Remark}

\numberwithin{equation}{section}

\begin{document}

\title[Almost commuting unitaries]{Almost commuting unitaries with spectral gap are near commuting unitaries}

\author{Tobias J. Osborne}
\address{Department of Mathematics, Royal Holloway, University of London, Egham, TW20 0EX, UK}
\email{tobias.osborne@rhul.ac.uk}

\subjclass[2000]{Primary 15A15, 15A27, 47A55; Secondary 47B47}

\date{\today}

\begin{abstract}
Let $\mathcal{M}_n$ be the collection of $n\times n$ complex
matrices equipped with operator norm. Suppose $U, V \in
\mathcal{M}_n$ are two unitary matrices, each possessing a gap
larger than $\Delta$ in their spectrum, which satisfy $\|UV-VU\| \le
\epsilon$. Then it is shown that there are two unitary operators $X$
and $Y$ satisfying $XY-YX = 0$ and $\|U-X\| + \|V-Y\| \le
E(\Delta^2/\epsilon)
\left(\frac{\epsilon}{\Delta^2}\right)^{\frac16}$, where $E(x)$ is a
function growing slower than $x^{\frac{1}{k}}$ for any positive
integer $k$.
\end{abstract}

\maketitle

\section{Introduction}

An old problem from the 1950s, popularised by Halmos, asks: if a
pair $\{A, B\}$ of matrices almost commute then are they necessarily
close to a pair $\{A', B'\}$ of commuting matrices
\cite{szarek:1990a, davidson:1985a, berg:1987a, halmos:1976a,
rosenthal:1969a}? Voiculescu realised that this is not necessarily
the case and presented a family of pairs of unitary matrices which
asymptotically commute, but which were far from any commuting pair
of matrices \cite{voiculescu:1983a}.

Since this time there has been considerable work on this problem
culminating in 1995 with the proof of Lin that for any pair $\{A,
B\}$ of \emph{hermitian} matrices which satisfy $\|AB-BA\| \le
\epsilon$, with $\epsilon >0$ and $\|A\|, \|B\| \le 1$, there exists
a hermitian pair $\{A', B'\}$ of matrices and $\delta(\epsilon) > 0$
such that $\|A'-A\| + \|B'-B\| \le \delta(\epsilon)$
\cite{lin:1995a} (see \cite{friis:1996a} for a simplified
exposition). The proof was nonconstructive and the dependence of
$\delta$ on $\epsilon$ was not quantified, apart from the nontrivial
fact that $\delta$ did not depend on the size of the matrices.
Recently, Hastings \cite{hastings:2008b} has presented a new
constructive proof of the result of Lin and was able to calculate
quantitative bounds.

In light of the recent results of \cite{hastings:2008b} it is
worthwhile to reconsider the problem of when a pair of almost
commuting unitary matrices are near a pair of commuting unitaries.
There have been several studies of this problem, partially motivated
by Voiculescu's original counterexample, and it was quickly realised
that there was, in general, a $K$-theoretic obstruction
\cite{exel:1989a, exel:1991a}. When this obstruction vanishes it was
shown that, indeed, a pair of almost commuting unitaries is close to
a commuting pair \cite{lin:1997a, lin:1998a}. Again, the proof was
nonconstructive, and quantitative bounds not provided. In this paper
we partially ameliorate this situation by proving the following

\begin{theorem}\label{thm:approxucomm}
Let $U$ and $V$ be two unitary operators such that
\begin{equation}
\|[U, V]\| \le \epsilon,
\end{equation}
where $[U, V] = UV-VU$, and also such that there exists $0 <
\Delta_1 < \pi$ and $0 < \Delta_2 < \pi$ with\footnote{We denote the
unit circle in the complex plane by $S^1$. Where obvious we identify
subintervals of $[-\pi,\pi)$ with their image in $S^1$ under the map
$e^{i\phi}$.}
\begin{equation}
\spec(U) \subset {S}^1\setminus [-\Delta_1, \Delta_1], \quad
\mbox{and}\quad \spec(V) \subset {S}^1\setminus [-\Delta_2,
\Delta_2].
\end{equation}
Then there exist two commuting unitary operators $X$ and $Y$ such
that
\begin{equation}
\|U-X\| \le \delta(\epsilon/(\Delta_1\Delta_2)), \quad \mbox{and}
\quad \|V-Y\| \le \delta(\epsilon/(\Delta_1\Delta_2))
\end{equation}
where $\delta(x) = E(1/x)x^{\frac16}$ and $E(x)$ grows slower than
$x^{1/k}$ for any positive integer $k$.
\end{theorem}

The proof of theorem~\ref{thm:approxucomm} is based on three
observations. The first is that the principle branch of the matrix
logarithm of a unitary matrix $U$ with spectral gap can be
represented in terms of a rapidly convergent Laurent series in $U$.
The second observation is that the matrix logarithms $\{A, B\}$ of a
pair $\{U, V\}$ of approximately commuting unitaries must themselves
be approximately commuting, so they are close to a pair of commuting
matrices $\{A', B'\}$. The final observation is that these commuting
matrices give rise to a commuting pair $\{X, Y\}$ of unitaries close
to $\{U, V\}$.

\section{Logarithms of unitary matrices with spectral gap}

\begin{theorem}\label{thm:convlog}
Let $U\in \mathcal{M}_n$ be a unitary matrix such that there exists
a gap $\Delta$ with $0 < \Delta < \pi$ such that
\begin{equation}
\spec(U) \subset S^1 \setminus [-\Delta, \Delta].
\end{equation}
Then the principle branch of the matrix logarithm of $U$ may be
represented as a Laurent series
\begin{equation}
\log(U) = -i\sum_{k=-\infty}^\infty c_k U^k
\end{equation}
where $|c_k| \le \min\{\pi, C/(\Delta k^4)\}$ and $C$ is constant.
\end{theorem}

\begin{remark}
This result is adapted from a physical argument developed in
\cite{osborne:2006d}.
\end{remark}

\begin{proof}
We begin by calculating the eigenvalues and eigenvectors for $U$:
\begin{equation}
Uv_j = e^{i\phi_j}v_j,
\end{equation}
where $v_j$ are the eigenvectors of $U$ and we choose $\phi_j \in
[0, 2\pi)$. (If the gap in $U$'s spectrum was not centred on $0$ we
could, by multiplying by an overall phase $e^{i\zeta}\mathbb{I}$,
$\zeta\in\mathbb{R}$, arrange for the zero of angle to to lie at the
origin. Such a gap always exists for finite dimensional unitary
matrices, but not necessarily for infinite operators.)

We want to find a hermitian matrix $H$ so that $U = e^{iH}$ and
$Hv_j = \phi_j v_j$. To do this we suppose that
\begin{equation}\label{eq:hseries}
H = \sum_{k=-\infty}^{\infty} d_k U^k,
\end{equation}
and we solve for the coefficients $d_k$: by applying both sides of
the above equation to the common eigenvector $v_j$ we find that
\begin{equation} \phi_j = \sum_{k=-\infty}^{\infty} d_k
e^{ik\phi_j}.
\end{equation}
Hence, if we can find $d_k$ such that
\begin{equation}\label{eq:ckseries}
\theta = \sum_{k=-\infty}^{\infty} d_k e^{ik\theta},
\end{equation}
for all $\theta \in [0, 2\pi)$ then we are done. (Recall that we've
arranged it so there are no eigenvalues of $U$ at the point $\theta
= 0$.) To solve for $d_k$ we integrate both sides of
(\ref{eq:ckseries}) with respect to $\theta$ over the interval $[0,
2\pi)$ against $\frac{1}{2\pi}e^{-il\theta}$, for $l\in\mathbb{Z}$:
\begin{equation}\label{eq:ckseries2}
\frac{1}{2\pi}\int_{0}^{2\pi}\theta e^{-il\theta}\,d\theta =
\frac{1}{2\pi}\sum_{k=-\infty}^{\infty} d_k \int_{0}^{2\pi}
e^{i(k-l)\theta}\,d\theta.
\end{equation}
Thus we learn that the $d_k$ are nothing but the fourier
coefficients of the periodic sawtooth function $f(\theta+2\pi l) =
\theta$, $\theta \in [0, 2\pi)$, $l\in\mathbb{Z}$:
\begin{equation}\label{eq:cns}
d_k = \begin{cases}\pi, \quad k=0 \\ \frac{i}{k}, \quad
k\not=0.\end{cases}
\end{equation}

Unfortunately the sawtooth wave has a jump discontinuity and hence
the fourier series is only conditionally convergent. The way to
proceed is to assume that we have some further information, namely,
that $U$ has a gap $\Delta$ in its spectrum.

The idea now is to exploit the existence of the gap to provide a
more useful series representation for $H$. We do this by calculating
the fourier coefficients $c_k$ of the sawtooth wave $f(\theta)$
convolved with a sufficiently smooth smearing function
$\chi_\gamma(\theta)$; the fourier series inherits a better
convergence from the smoothness properties of the smearing function.
That is, we define $c_k$ to be the fourier coefficients of
\begin{equation}
g(\theta) = (f\star \chi_\gamma)(\theta) = \int_{-\infty}^{\infty}
f(\theta-y)\chi_\gamma(y)dy.
\end{equation}

We choose $\chi_\gamma$ to be the function
\begin{equation}
\chi_\gamma(x) = \begin{cases}
\left(1-\left(\frac{x}{\gamma}\right)^2\right)^3, \quad &|x| \le
\gamma
\\ 0, \quad &|x| > \gamma.\end{cases}
\end{equation}
The fourier transform $\widehat{\chi}_\gamma(t) \equiv
\frac{1}{\sqrt{2\pi}}\int_{-\infty}^{\infty} \chi_\gamma(x) e^{ixt}
\,dx$ satisfies
\begin{equation}\label{eq:chidecay}
|\widehat{\chi}_\gamma(t)| \le \frac{C}{\gamma t^3},
\end{equation}
where $C$ is a constant. Note that, as a consequence of the compact
support of $\chi_\gamma(y)$, $g(\theta) = f(\theta)$, $\forall\theta
\in (\gamma, 2\pi -\gamma)$. An application of the convolution
theorem then tells us that the fourier coefficients $d_k$ are given
by
\begin{equation}
c_k = \widehat{\chi}_\gamma(k) d_k.
\end{equation}

Choosing $\gamma < \Delta$ allows us to conclude that
\begin{equation}
H = \sum_{k=-\infty}^{\infty}  c_k U^k
\end{equation}
because both $f(\theta)$ and $g(\theta)$ agree on the spectrum of
$U$. According to (\ref{eq:chidecay}) we now have that $|c_k| \le
\min\{\pi, C/(\Delta k^4)\}$.
\end{proof}

\section{Constructing commuting logarithms}
In this section we exploit the following theorem of Hastings.
\begin{theorem}[\cite{hastings:2008b}]\label{thm:approxcomm}
Let $A$ and $B$ be hermitian, $n\times n$ matrices, with $\|A\|$,
$\|B\| \le 1$. Suppose $\|[A,B]\| \le \epsilon$. Then there exist
hermitian $n\times n$ matrices $A'$ and $B'$ such that
\begin{enumerate}
\item $[A',B'] = 0$, and
\item $\|A'-A\| \le \delta(\epsilon)$ and $\|B'-B\| \le
\delta(\epsilon)$, with
\begin{equation}
\delta(\epsilon) = E(1/\epsilon)\epsilon^{\frac16},
\end{equation}
\end{enumerate}
where the function $E(x)$ grows slower than $x^{\frac1k}$ for any
positive integer $k$. The function $E(x)$ does not depend on $n$.
\end{theorem}

\begin{proof}[Proof of theorem~\ref{thm:approxucomm}]
The first step of the proof is to exploit the presence of the
spectral gap and use theorem~\ref{thm:convlog} to represent the
logarithms $A$ and $B$ of $U = e^{iA}$ and $V = e^{iB}$ via Laurent
series with fast decay:
\begin{equation}
A = \sum_{j = -\infty}^{\infty} c_j U^j
\end{equation}
and
\begin{equation}
B = \sum_{k = -\infty}^{\infty} d_k V^k,
\end{equation}
with
\begin{equation}
|c_j| \le \min\{\pi, C/(\Delta_1j^4)\}, \quad |d_k| \le \min\{\pi,
C/(\Delta_2 k^4)\},
\end{equation}
where $C$ is a constant.

Now that we have an expression for $A$ and $B$ we can work out the
norm of their commutator using the Leibniz property via
\begin{equation}
\begin{split}
[A,B] &= \sum_{j=-\infty}^\infty\sum_{k=-\infty}^\infty c_jd_k [U^j,
V^k] \\
&=\sum_{j=-\infty}^\infty\sum_{k=-\infty}^\infty\sum_{m =
1}^{j}\sum_{n = 1}^{k} c_jd_k V^{m-1}U^{n-1}[U, V]U^{j-n}V^{k-m},
\end{split}
\end{equation}
and then taking the norm of both sides. Applying the triangle
inequality and the unitary invariance of the operator norm gives us
\begin{equation}
\begin{split}
\|[A,B]\| &\le \sum_{j=-\infty}^\infty\sum_{k=-\infty}^\infty
 jk|c_j||d_k| \|[U, V]\| \le \epsilon
\left(\sum_{j=-\infty}^\infty
j|c_j|\right)\left(\sum_{k=-\infty}^\infty
 k|c_k|\right) \\
 &\le \epsilon \frac{\alpha}{\Delta_1\Delta_2},
\end{split}
\end{equation}
where $\alpha$ is a constant coming from $C^2$ and the summation of
the separate infinite series.

The next step is to apply theorem~\ref{thm:approxcomm} to construct
two commuting operators $A'$ and $B'$ such that
\begin{equation}
\|A'-A\| \le \delta(\epsilon {\alpha}/{\Delta_1\Delta_2}),
\quad\mbox{and}\quad \|B'-B\| \le \delta(\epsilon
{\alpha}/{\Delta_1\Delta_2}).
\end{equation}
and then, via exponentiation, we define $U' = e^{iA'}$ and $V' =
e^{iB'}$. The distance between $U'$ and $U$ can be bounded as
follows (following \cite{bratteli:1997a}, p.\ 252):
\begin{equation}
\|U'-U\| \le \int_0^1 \|A'-A\| ds \le \epsilon,
\end{equation}
and similarly for $V'$. Redefining $E(x)$ gives the result.
\end{proof}

\bibliographystyle{amsplain}

\providecommand{\bysame}{\leavevmode\hbox
to3em{\hrulefill}\thinspace}
\providecommand{\MR}{\relax\ifhmode\unskip\space\fi MR }
\providecommand{\MRhref}[2]{%
  \href{http://www.ams.org/mathscinet-getitem?mr=#1}{#2}
} \providecommand{\href}[2]{#2}

\end{document}